\documentclass{article}
\usepackage[utf8]{inputenc}
\usepackage{amsfonts} 
\usepackage{amssymb}
\usepackage{amsmath}
\usepackage{graphicx}%
\usepackage{xcolor}
\usepackage{tikz-cd}
\usepackage{comment}
\usepackage{txfonts}
\providecommand{\U}[1]{\protect\rule{.1in}{.1in}}
\newtheorem{theorem}{Theorem}[section]

\newtheorem{definition}[theorem]{Definition}

\newtheorem{lemma}[theorem]{Lemma}

\newtheorem{proposition}[theorem]{Proposition}

\newcommand{\N}{\mathbb{N}} 
\newcommand{\K}{\mathcal{K}}
\newcommand{\D}{{\mathbf D}}
\newcommand{\E}{{\mathbf E}}

\newcommand{\C}{\mathcal{C}}

\newcommand{\F}{\mathcal{F}}
\newcommand{\G}{\mathcal{G}}

\newcommand{\cR}{\mathcal{R}}
\newcommand{\cS}{\mathcal{S}}
\newcommand{\cP}{\mathcal{P}}
\newcommand{\cQ}{\mathcal{Q}}
\newcommand{\dubCan}{2^{ \N \times \N}}
\newcommand{\lo}{{\bf LO^*}}
\newcommand{\wf}{{\bf WF^*}}

\newcommand{\leex}{{\le ^*_x}}
\newcommand{\lex}{{< ^*_x}}
\newenvironment{proof}[1][Proof]{\noindent\textbf{#1.} }{\ \rule{0.5em}{0.5em}}

\title{A note on derivatives, expansions and $\Pi^1_1$-ranks}
\author{Udayan B. Darji and Felipe García-Ramos}
\date{}

\begin{document}

\maketitle
\abstract{$\Pi_1^1$-ranks are a natural tool for studying coanalytic sets in descriptive set theory. In \cite{Kechris}, Kechris provided a technique to build $\Pi_1^1$-ranks using derivatives. In this note we will prove a variant of this result that is applicable to the $\Gamma$-rank. Some dynamical ranks, like the entropy rank can be stated in terms of the $\Gamma$-rank.}

\section{Introduction}

At the dawn of descriptive set theory, Lebesgue made an infamous mistake by assuming that the continuous projection of a Borel set was Borel. Suslin spotted this mistake 10 years later and began the study of \emph{analytic} or $\Sigma_1^1$ sets. One of tools which was developed later to understand the complexity of Borel sets and the difference between Borel and co-analytic sets were $\Pi_1^1$-ranks. One of the most well known $\Pi_1^1$-ranks is the Cantor-Bendixson rank. Kechris developed a very general set up were, using the concept of derivatives (or the dual version of expansions), one can prove that the Cantor-Bendixson rank, as well as several other natural ranks, are $\Pi_1^1$ (see Theorem \ref{BDerivatives})\cite{Kechris}. 

A natural rank that appears in dynamics as well as other areas of mathematics is the $\Gamma$-rank on product spaces. Though the $\Gamma$-rank can be stated in terms of expansions it does not fit exactly in the context of the dual of Theorem \ref{BDerivatives}. In this note, we adapt the proof given in \cite{Kechris} in order to show that the $\Gamma$-rank is a $\Pi_1^1$-rank. 
A concrete example of the $\Gamma$-rank is the entropy rank for topological dynamical systems that was introduced by Barbieri and the second author \cite{barbieri2020} to classify dynamical systems with completely positive entropy.  

We note that, using effective descriptive set theory, Westrick recently proved that the entropy (or TCPE) rank is an effective $\Pi_1^1$-rank \cite[Corollary 2]{westrick2019topological}. It is possible to transfer this result to the classic descriptive theory setting, nonetheless, the approach of this paper gives a direct proof using only classical descriptive set theory.   

\textbf{Acknowledgment:} The authors would like to thank Dominik Kwietniak,
Slawomir Solecki and Linda Westrick for motivating conversations. 
The second author was supported by the CONACyT grant 287764. 

\section{$\Pi^1_1$-ranks}
In this section we present the necessary basic definitions and background results concerning $\Pi^1_1$-ranks. We also prove a new result concerning Borel expansions. 

Recall that a subset of a Polish space is \textbf{analytic} or $\Sigma^1_1$ if it is the continuous image of a Borel set of a Polish space. Complements of $\Sigma^1_1$ sets are $\Pi^1_1$ or \textbf{coanalytic}.

\begin{definition}
Let $C$ be a set. A {\bf rank} on $C$ is a function $\varphi:C\rightarrow\omega_1$, where $\omega_1$ is the set of countable ordinals. Associated with $\varphi$ we have the relations $<_{\varphi}$ and $\le_{\varphi}$ defined as follows:
\[ x <_{\varphi} y \iff \varphi(x) < \varphi (y)
\]
\[ x \le_{\varphi} y \iff \varphi(x) \le \varphi (y).
\]
\end{definition}
\begin{definition}\label{pi11}
Let $X$ be a Polish space, $C \subseteq X$ and $\varphi : C \rightarrow \omega_1$ a rank on $C$. We say that $\varphi$ is a  {\bf  $\Pi^1_1$-rank} if $C$ is $\Pi^1_1$ and  there are relations $P, Q \subseteq X^2$, with one of them $\Sigma^1_1$ and the other $\Pi^1_1$, such that for all $y \in C$ we have that 
\begin{gather*}
    \{x \in C: \varphi (x) \le \varphi(y)\} = \{x\in X: (x,y) \in P\} = \{x \in X: (x,y) \in Q\}.
\end{gather*}
Loosely speaking, $\varphi$ is a $\Pi^1_1$-rank if $\{x: \varphi (x) \le \varphi(y)\} $ is "uniformly Borel in $y$".
\end{definition}
We will use the following reformulation of $\Pi^1_1$-rank in our proof.
\begin{proposition}\label{altpi11}\cite[Exercise 34.3]{Kechris} Let $X, C, \varphi$ as in Definition~\ref{pi11}. Then, $\varphi$ is a $\Pi^1_1$-rank if and only if there  are $\Sigma^1_1$ relations $P, Q \subseteq X^2$ such that for all $y \in C$ we have that
    \begin{align*}
    & \{x \in C: \varphi (x) \le \varphi(y)\} = \{x\in X: (x,y) \in P\},\text{ and} \\ 
    & \{x \in C: \varphi (x) < \varphi(y)\}= \{x \in X: (x,y) \in Q\}.
\end{align*}
\end{proposition}
The following are fundamental results on $\Pi^1_1$-ranks \cite{Kechris}.
\begin{theorem}
Every $\Pi^1_1$ set admits a $\Pi^1_1$-rank.
\end{theorem}
\begin{theorem}
Let $C$ be a $\Pi^1_1$ set and $\varphi$ be a $\Pi^1_1$-rank on $C$. If $A \subseteq C$ is $\Sigma^1_1$, then $\varphi$ is bounded on $A$, i.e., there exists $\alpha < \omega_1$ such that $\varphi(x) < \alpha$ for all $x \in A$. In particular,

\[ C \mbox{ is Borel} \iff \varphi \mbox{ is bounded on } C.
\]
\end{theorem}

We next recall the  notion of derivatives and how it induces $\Pi^1_1$-ranks in a natural way \cite[Section~34.D]{Kechris}.
Let $\K(X)$ denote the space of all compact subsets of $X$ endowed with the Hausdorff metric. 
\begin{definition} 
A map $\D:\K(X) \rightarrow \K(X)$ is a {\bf derivative} if the following holds:
\[ \D(A) \subseteq A \ \ \ \ \& \ \ \ A \subseteq B \implies \D(A) \subseteq \D(B). \] 

\end{definition}
 Derivatives appear in a variety of contexts and they induce $\Pi^1_1$-ranks in a natural way. For  a derivative $\D$, let 
 \begin{align*}
    &\D ^{0} (A)  =  A\\
    &\D ^{\alpha+1}  =  \D (\D^{\alpha}(A))\\
    &\D ^{\lambda} (A) = \cap _{\beta < \lambda} \D^{\beta}(A) \textit { if } \lambda \textit{ is a limit ordinal.}
\end{align*}
Let $A \in \K(X)$. Then, there exists a countable ordinal $\alpha$ such that $\D^{\alpha} = \D^{\alpha+1}$. Such an ordinal exists since in a separable metric space a chain of strictly decreasing sequence of closed sets must be countable. We let $|A|_{\D}$ be the least such $\alpha$. Moreover, we let $\D^{\infty} (A)=\D^{|A|_{\D}}$, i.e., the stable part of $A$. 

A useful classical Borel derivative is the Cantor-Bendixson derivative given by 
\[ A \rightarrow A' \] 
where $A'$ is the set of limit-points of $A$ \cite[Theorem 6.11]{Kechris}. The $\alpha^{th}$ Cantor-Bendixson derivative of $A$ is denoted by $A^{\alpha}$, $|A|_{CB}$ denotes least ordinal $\alpha$ such that $A^{\alpha +1} = A^{\alpha}$, and  $A^{\infty} = A^{|A|_{CB}}$, i.e.,  the stable part of $A$.

The following is an important theorem which relates derivatives to $\Pi^1_1$-ranks.
\begin{theorem}\label{BDerivatives} \cite[Theorem 34.10]{Kechris}
Let $\D:\K(X) \rightarrow \K(X)$ be a Borel derivative and 
\[ \C = \{A \in \K(X): D^{\infty}(A) = \emptyset\}.\]
Then, $\C$ is $\Pi^1_1$ and $\varphi : \C \rightarrow \omega_1$ defined by $\varphi(A)= |A|_{\D}$ is a $\Pi^1_1$-rank on $\C$.
\end{theorem}

A dual notion of derivatives is the concept of expansion. 
\begin{definition}
A map $\E:\K(X) \rightarrow \K(X)$ is an {\bf expansion}  means that \[A \subseteq \E(A) \ \ \ \ \& \ \ \  A \subseteq B \implies \E(A) \subseteq \E(B)\]
For an expansion $E$, as earlier, we let 
\begin{align*}
    &\E ^{0} (A) =  A\\
    &\E ^{\alpha+1}  =  \E (\E^{\alpha}(A))\\
    &\E ^{\lambda} (A) = \overline{\cup _{\beta < \lambda} \E^{\beta}(A)} \textit { if } \lambda \textit{ is a limit ordinal.}
\end{align*}
We let $|A|_{\E}$ be the least such $\alpha$ such that $\E ^{\alpha+1}  =  \E^{\alpha}(A)$.  Moreover, we let $\E ^{\infty} (A)=\E ^{|A|_{\E}}$, i.e., the stable part of $A$.
\end{definition}
For every expansion one can define a derivative (and vice-versa). Furthermore, one can formulate the above Theorem \ref{BDerivatives} in terms of expansions. We will prove a variant of this dual.

\begin{theorem}\label{BorelExp}
Let $X$ be a compact metric space and $\E$ be a Borel expansion on $\K(X)$ and let 
\[ \C= \{A \in K(X): \E^{\alpha} (A) =X \textit { for some } \alpha\}.
\]
Then, $\C$ is $\Pi^1_1$ and $\varphi : \C \rightarrow \omega_1$ defined by $\varphi(A)= |A|_\E$ is a $\Pi^1_1$-rank on $\C$.
\end{theorem}

Before proving the theorem, let us show a specific instance of a Borel expansion, the $\Gamma$ map.

\begin{definition}
Let $X$ be a compact metric space and $E\subseteq X^{2}$ a closed set.
We define $E^+$ as the smallest equivalence relation that contains
$E$ and $\Gamma(E)=\overline{E^+}$. For an ordinal $\alpha$, $\Gamma^{\alpha}(E)$ is defined by
\[
\Gamma^{\alpha}(E) = \Gamma( \Gamma^{\alpha -1}(E )) 
\]
if $\alpha$ is the successor ordinal and
\[
\Gamma^{\alpha}(E) = \overline {\cup_{\beta < \alpha} \Gamma^{\beta}(E)}
\]
if $\alpha$ is a limit ordinal.  
\end{definition}

Recall that in a topological  space with countable basis, a chain of strictly increasing sequence of closed sets must be countable. From this we have the following.
\begin{proposition}
Let $X$ be a compact metrizable space and $E\subset X$. There exists a countable ordinal $\alpha$ such that $\Gamma^{\alpha}%
(E)=\Gamma^{\alpha+1}(E)$.
\end{proposition}
The smallest ordinal that satisfies the statement in the previous proposition is called the $\Gamma$-\textbf{rank of} $E$.

Before proving that $\Gamma:K(X\times X)\rightarrow K(X\times X)$ we will prove a lemma. 
\begin{lemma}\label{BorelLimit}
Let $X$ be a compact metrizable space, $\varphi_n: K(X) \rightarrow K(X)$ be a Borel map, $n \in \N$, and $\varphi:K(X) \rightarrow K(X)$ defined by
\[ \varphi(A) := \overline{\cup_{n=1}^{\infty} \varphi_n (A)}
\]
Then, $\varphi$ is Borel.
\end{lemma}
\begin{proof}
Define $\psi _n :K(X) \rightarrow K(X)^n$ by 
\[ \psi _n (A) := (\varphi_1(A), \ldots, \varphi_n(A)).
\]
Then, $\psi_n$ is Borel. Moreover, as the union map is continuous, we have that, for each $n \in \N$, $A \rightarrow \cup_{i=1}^n  \varphi_i(A)$ is Borel. Now $\varphi$ is simply the pointwise limit of these maps and hence itself Borel. 
\end{proof}
\begin{proposition}\label{GammaBorel}
Let $X$ be a compact metric space. Then, 
\[
\Gamma : K(X \times X) \rightarrow K (X \times X)
\]
is a Borel map. 
\end{proposition}
\begin{proof}
We first note that $A \rightarrow A^{+}$ is a continuous map. Define $\sim_n: K (X \times X) \rightarrow K (X \times X)$ by 
$\sim_n (A) :=\{(x,y):\exists x=x_0, \ldots x_n =y \mbox { such that } (x_i, x_{i+1}) \in A \ \  \forall \ 0 \le i <n\}$. We note that $\sim_n$ is a continuous map. Hence, the map $A \rightarrow \sim_n(A^{+})$ is continuous. Now we have that 
\[ \Gamma(A) = \overline{ \cup_{n=1}^{\infty} \sim_n (A^{+})}
\]
is Borel by Lemma~\ref{BorelLimit}.
\end{proof}\\

We now proceed to prove Theorem~\ref{BorelExp}. We follow the general outline of \cite[Theorem~34.10]{Kechris} adapted to this set up. 

\begin{proof}[Proof of Theorem~\ref{BorelExp}]
It suffices to show that $|\cdot|_{\E}$ is a $\Pi^1_1$-rank on $\C \setminus \{X\}$.

We first show that $\C$ is $\Pi^1_1$. As $\E$ is Borel, $gr(\E)$, the graph of $E$, is also Borel. Let $\Delta$ be the diagonal of $K(X) \setminus \{X\}$. Then, $gr(\E)\cap \Delta$ is Borel. As
\[\F = \{A \in K(X): \E(A)=A  \ \&\  A \neq X\}\] is the 1-1 projection of  the Borel set $gr(\E)\cap \Delta$, we have that $\F$ is Borel. Hence,
\[ \G = \{(A,B) \in K(X) \times K(X): B \in \F \ \ \& \ \ A   \subseteq B\}
\]
is Borel as it is the intersection of two sets, one closed and the other Borel, namely, 
\[\{(A,B) \in K(X) \times K(X): A   \subseteq B\} \ \ \ \  \ K(X) \times \F.
\]
As the projection of Borel sets are $\Sigma^1_1$, and  \[A \notin \C \Longleftrightarrow  (A,B) \in \G \textit{ for some B},
\]
we have that $K(X) \setminus \C$ is $\Sigma^1_1$, or, equivalently, $\C$ is $\Pi^1_1$.

We proceed to construct required $\Sigma^1_1$ sets as  in Proposition~\ref{altpi11}. In order to do this we need a  a variant of a standard combinatorial  $\Pi^1_1$ set as in the proof of \cite[Theorem~34.10]{Kechris}. We recall the basic terminology.

For $x \in \dubCan$ we let \[D^*(x) = \{m \in \N: x(m,m) =1\}\] and we define \[m \leex n \Leftrightarrow [m, n \in D^*(x) \ \&  \ x(m,n) =1]. \]
We let $\lo$ be the set of all $x \in \dubCan$ such that $\leex$ is a linear order, $0 \in D^*(x)$ and $0 \leex m$ for all $m \in D^*(x)$ and let $\wf$ be the set of all $x \in \lo$ such that $\leex$ is a wellordering. It is known that $\lo$ is a closed subset of $\dubCan$ and $\wf$ is a $\Pi^1_1$-complete subset of $\dubCan$ and $x \mapsto |x|^*$ is a $\Pi^1_1$ rank on $\wf$ where $|x|^*$ is the order type of $x \in \wf$. Moreover, the range of $\wf$ is $\omega_1 \setminus \{0\}$.

We next show that it suffices to construct $\Sigma^1_1$ subsets $\cR$ and $\cS$ of $\lo \times K(X)$ which satisfy the following properties:
\begin{align*}
   & \forall A \in \C \setminus \{X\},  & \{x \in \lo : (x, A) \in \cR \} = \{x \in \wf: |x|^* \le |A|_{\E}\}  \tag{R}  \\
     & \forall x \in \wf, &  \{A \in K(X) : (x, A) \in \cS \} = \{A \in \C : |x|^* = |A|_{\E}\}. \tag{S}
\end{align*}
Indeed, let 
\[ \cP = \{(A,B)  \in K(X) ^2 : \exists x \in \lo \textit{ such that } (x,B) \in \cR \ \& \ (x,A) \in \cS\}. \]
Then, $\cP $ is $\Sigma^1_1$ and for all $B \in \C \setminus \{X\}$ we have that 
\[ \{A \in \C \setminus \{X \}: |A|_{\E}\le |B|_{\E}\} = \{A \in \K(X): (A, B) \in \cP\}.\]
Indeed, the containment $\subseteq$  of the above equality is clear. To see the containment $\supseteq$, let $(A,B) \in \cP$ and $x\in \lo$ be such that $(x,B) \in \cR$ and $(x,A) \in \cS$. Applying  Condition (R) to our set $B \in \C \setminus \{X\}$, we have  that $x \in \wf$ and $|x|^* \le |B|_{\E}$. As $x \in \wf$ and $(x,A) \in \cS$, by Condition (S) we have that $A \in \C $ and $|x|^* = |A|_{\E}$. As $|x|^* >0$, we have that $A \neq X$. Hence, we have that $A \in \C \setminus \{X\}$ with $|A|_{\E} \le |B|_{\E}.$

In order to obtain $\cQ$, we choose a Borel function $x \mapsto x'$ from $\lo$ to $\lo$ such that $|x'|^* = |x|^* +1$ and $x \in \wf$ iff $x' \in \wf$.
We let 
\[ \cQ = \{(A,B)  \in K(X) ^2 : \exists x \in \lo \textit{ such that } (x',B) \in \cR \ \& \ (x,A) \in \cS\}. \]
As $x \mapsto x'$ is Borel, we have that $\cQ $ is $\Sigma^1_1$. Moreover for all $B \in \C \setminus \{X\}$ we have that 
\[ \{A \in \C \setminus \{X \}: |A|_{\E} < |B|_{\E}\} = \{A \in \K(X): (A, B) \in \cQ\}.
\]
Indeed, the containment $\subseteq$  of the above equality is clear. To see the containment $\supseteq$, let $(A,B) \in \cQ$ and $x\in \lo$ be such that $(x',B) \in \cR$ and $(x,A) \in \cS$. Applying  Condition (R) to our set $B \in \C \setminus \{X\}$,  we have that $x' \in \wf$ and $|x'|^* \le |B|_{\E}$. As $x' \in \wf$, we have that $x \in \wf$. Now, as $x \in \wf$ and  $(x,A) \in \cS$, by Condition (S) we have that $A \in \C $ and $|x|^* = |A|_{\E}$. As $|x|^* >0$, we have that $A \neq X$. Hence, we have that $A \in \C \setminus \{X\}$ with $|A|_{\E} = |x|^* < |x'|^* \le |B|_{\E}$, i.e., $|A|_{\E} < |B|_{\E}$

By Proposition~\ref{altpi11}, we have that $|\cdot|_{\E}$ is a $\Pi^1_1$-rank on $\C \setminus \{K\}$, provided that we construct $\Sigma^1_1$ sets $\cR,\cS$ with the required properties.

We define
\begin{align*}
  \cR = \{(x,A) \in \lo \times K(X) : \exists h \in K(X)^{\N} \textit{ s.t. } h(0) =A, \\  \&\ \  \forall m \in D^*(x), h(m) \neq X\\ \& \  m \in D^*(x) \setminus \{0\} \Rightarrow \overline{\cup_{n \lex m} \E(h(n))} \subseteq h(m) \}.  
\end{align*}
\begin{align*}
    \cS = \{(x,A) \in \lo \times K(X): \exists h \in K(X)^{\N} \textit{ s.t. }   h(0) =A,   \\  \&\ \  \forall m \in D^*(x), h(m) \neq X\\ \& \  m \in D^*(x) \setminus \{0\} \Rightarrow  \overline{\cup_{n \lex m} \E(h(n))} \subseteq h(m)\\
    \& \ \overline{\cup_{m \in D^*(x)} \E(h(m))} = X\}  
\end{align*}
By Proposition~\ref{BorelLimit} we have that $\overline {\cup _n}: K(X)^{\N} \rightarrow K(X)$ defined by $\overline {\cup _n} (A_n) = \overline {\cup _n (A_n)}$ is Borel. This fact together with the standard quantifier counting technique imply that $\cR$ and $\cS$ are $\Sigma^1_1$. 

Let us now observe that $\cR$ and $\cS$ satisfies Conditions (R) and (S), respectively.  

To see that $\cR$ satisfies Condition (R), let $A \in \C \setminus \{X\}$. That containment $\supseteq$ holds in Condition (R) follows directly from the definition of $\cR$. For the containment $\subseteq$, we note that if $m >0$ is in $D^*(x)$, then for some $\alpha < |A|_{\E}$ we have that $\overline{\cup_{n \lex m} \E(h(n))} \nsupseteq \E^{\alpha+1} (A)$. This is so, for otherwise, for all $\alpha < |A|_{\E}$ we would have that $ \E^{\alpha+1} (A) \subseteq \overline{\cup_{n \lex m} \E(h(n))} \subseteq h(m) \neq X$, implying that $\E^{\infty} (A)  = \overline{ \cup_{\alpha < |A|_{\E} }\E^{\alpha+1} (A) } \neq X$ and  that $A \notin \C$. Now we define $f$ from $\lex$ into $|A|_{\E}$. Define $f(0) =0$ and for $m \in D^*(x) \setminus \{0\}$, let 
\[ f(m) = \text{ the least } \alpha < |A|_{\E} \text{ such that } \overline{\cup_{n \lex m} \E(h(n))} \nsupseteq \E^{\alpha+1} (A). 
\]
It suffices to show that $f$ is order preserving, i.e., $m  \lex p $ implies $f(m) < f(p)$. Indeed,  this would imply that $x \in \wf$ and $|x|^* \le |A|_{\E}$. Finally, to see that $f$ is order preserving, let $m \lex p$, $m \neq 0$.  Then,
\[  \E^{f(m)}(A) = \overline{ \cup _{\alpha < f(m)} \E^{\alpha+1}(A) } \subseteq \overline{\cup_{n \lex m} \E(h(n))}.
\]
Hence, $\E^{f(m)} (A) \subseteq h(m)$, implying that $\E^{f(m)+1} (A) \subseteq \E(h(m))$. As $m \lex p$, we have that $\E^{f(m)+1} (A) \subseteq \E(h(m)) \subseteq \overline{\cup_{q \lex p} \E(h(q))}$, implying that $f(q) \ge f(m) +1$ and  exhibiting that $f$ is order preserving.

To see that $\cS$ satisfies Condition (S), let $x \in \wf$. That containment $\supseteq$ holds in Condition (S) follows directly from the definition of $\cS$. That containment $\subseteq$ holds follows from our order preserving function $f$ from $\lex$ into $|A|_{\E}$ and the fact that $\overline{\cup_{m \in D^*(x)} \E(h(m))} = X$.
\end{proof}

\section{Application: entropy rank}
  A set $I \subseteq \N$ has \textbf{positive density} if
$\liminf_n \frac{|I \cap [1,n]|}{n>0} >0$.

We say $(X,T)$ is a \textbf{topological dynamical system (TDS)} if $X$ is a compact metrizable space and $T:X\rightarrow X$ is a continuous function.
We say $(X_2,T_2)$ is a \textbf{factor} of $(X_1,T_1)$ if there exists a surjective continuous function $\phi : X_2\rightarrow X_1$ (called a \textbf{factor map}) such that $\varphi\circ T_{1}=T_{2}\circ\varphi$. 

 Given a TDS $(X,T)$ and $\{U,V\}\subset X$, we say $I \subset \N$ is an \textbf{independence set for $\{U,V\}$} if for all finite $J \subseteq I$,
and for all $(Y_j) \in \prod _{j\in J}{\mathcal A}$, we have that
$$\cap_{j\in J}T^{-j}(Y_j)\neq \emptyset.$$


Let $(X,T)$ be a TDS and $\mathcal{U},\mathcal{V}$ open covers of $X$. We denote the smallest cardinality of a subcover of $\mathcal{U}$ with $N(\mathcal{U})$, and
\[
\mathcal{U}\vee\mathcal{V}=\{U\cap V:U\in\mathcal{U}\text{ and }V\in\mathcal{V}\}.
\] We define the \textbf{entropy of
$(X,T)$ with respect to $\mathcal{U}$} as
\[
h_{\text{top}}(X,T,\mathcal{U})=\lim_{n\rightarrow\infty}\frac{1}{n }\log N(\vee^n_{m=1}T^{-m}(\mathcal{U})).
\]
The \textbf{(topological) entropy} of $(X,T)$ is defined as \[
h_{\text{top}}(X,T)=\sup_{\mathcal{U}}h_{\text{top}}(T,\mathcal{U}).
\]
\begin{definition}
A TDS has \textbf{complete positive entropy (CPE)}  if every non-trivial factor has positive entropy. 
\end{definition}
Let $X$ be a compact metrizable space and $C(X,X)$ be the set of all continuous functions from $X$ into $X$ endowed with the uniform topology. We will now define a subspace of $C(X,X)$.
\begin{definition}
Given a compact metrizable space $X$ we define 
\[
\text{CPE}(X)=\{T\in C(X,X): (X,T)\text{ has CPE}\}.
\]
\end{definition}

Local entropy theory was initiated in \cite{blanchard93}. For more information see the survey \cite{glasner2009local} or the book \cite{kerr2016ergodic}. 

\begin{definition}
Let $(X,T)$ be a TDS. We say that  $[x_1, x_2] \in X\times X$ is an \textbf{independence entropy pair (IE-pair) of $(X,T)$} if for every pair of open sets $A_1,A_2$, with $x_1\in A_1$ and $x_2\in A_2$, there exists an independence set for $\{A_1,
A_2\}$ with positive density. The set of IE-pairs of $(X,T)$ will be denoted by $E(X,T)$.
\end{definition}

	
	
	
	




We are particularly interested in studying the $\Gamma$-rank in the case when $E=E(X,T)$ is the set of independence
entropy pairs.

\begin{definition}
Let $(X,T)$ be a TDS. The $\Gamma$-rank of the set of entropy pairs is called the
\textbf{entropy rank }of $(X,T)$.
\end{definition}

An equivalent statement of the following result was proved in \cite{blanchard93} (also see \cite[Theorem 12.30]{kerr2016ergodic}).   
\begin{theorem}
A TDS has CPE if and only if $\Gamma^{\alpha}(E(X,T))=X^{2}$ where
$\alpha$ is the entropy rank of $(X,T)$.
\end{theorem}

The following proposition was proved in \cite{darjiGR}. We give a proof for completeness. 
\begin{proposition}\label{Borelext}
 Consider the mapping $E: C(X,X) \rightarrow K(X \times X)$ given by $E(T) = E(X,T)$. Then, $E$ is a Borel map.
\end{proposition}
\begin{proof}
Let $U,V$ be open in $X$. We first observe that
\[\{ T \in C(X,X):  \ E(T) \cap (U \times V)\neq \emptyset\}\tag{\(\dagger\)}
\]
is Borel. Indeed, using an equivalent definition of independence given in \cite[Lemma 3.2]{KerrLiMA} we have that $\dagger$ is satisfied by $T$ if and only if there is a rational number $ r >0 $ such that for all $l \in \N$ there is an interval $I \subseteq \N$ with $|I| \ge l$ and a finite set $F \subseteq I$ with $|F| \ge r |I|$ such that $F$ is an independent set for $(U,V)$.
It is easy to verify that for fixed $U,V, r, l, I, F$ set
\[
   \{T \in C(X,X): F \text{ is an independent set for } (U,V) \text{ for } T  \tag{\(\ddagger\)} \}  
\]
is open. Now the set in $\dagger$ is the result of a sequence of  countable union and countable intersections of sets of type $\ddagger$. Hence, $\dagger$ is Borel. Since $X$ has a countable basis, by taking unions, we have that $\dagger$ is Borel when $U \times V$ is replaced by any open set $W \subseteq X \times X$. Every closed set in $X \times X$ is the monotonic intersection of a sequence of open sets in $X \times X$. This and the fact that $E(T)$ is closed imply that $\dagger $ is Borel when $U \times V$ is replaced by a closed set $C \subseteq X \times X$. Reformulating the last statement, we have that for all open $W \in X \times X$, the set 
\[\{ T \in C(X,X):  \ E(T) \subseteq W \}\tag{\(\diamond\)}
\]
is Borel. Putting $\dagger$ and $\diamond$ together, we have that 
\[\{T \in C(X,X):  E(T) \subseteq \cup_{i=1}^n  (U_i \times V_i) \ \  \&  \  E(T) \cap (U_i \times V_i) \neq \emptyset, 1 \le i \le n \}
\]
is Borel whenever $U_1, \ldots,U_n, V_1, \ldots V_n$ are open in $X$, completing proof.

\end{proof}

\begin{theorem}\label{mainpi11}
Let $X$ be a compact metrizable space. Then, $\varphi: C(X,X) \rightarrow \omega_1$ defined by $\varphi (T) =  |E(X,T)|_{\Gamma}$ is a $\Pi^1_1$-rank on $C(X,X)$.
\end{theorem}
\begin{proof}
By Proposition~\ref{GammaBorel}, we have that $\Gamma$ is Borel. Setting $\E = \Gamma$ in Theorem~\ref{BorelExp}, we have that the map that takes $A \in K(X \times X)$ to $|A|_{\Gamma}$ is a $\Pi^1_1$-rank on the set $ \C ' = \{ A \in K(X \times X): \Gamma ^{|A|_{\Gamma}} (A) = X \}$.  By the definition of $\Pi^1_1$ rank, there are sets $P, Q \in K(X \times X ) ^2$, as in the Definition~\ref{pi11}, which verify that $ | \cdot |_{\Gamma}$ is a $\Pi^1_1$ rank on $\C '$. By Lemma~\ref{Borelext}, we have that $E$ is Borel and hence $(E \times E)^{-1}(P)$ and  $(E \times E)^{-1}(Q)$, are $\Sigma^1_1$, and $\Pi^1_1$ in $C(X,X) \times C(X,X)$, respectively. This and the fact that $E^{-1}(\C') = C$, we have that $(E \times E)^{-1}(P)$ and  $(E \times E)^{-1}(Q)$ exhibit that $\varphi$ is a $\Pi^1_1$-rank on $\C$.  
\end{proof}

Examples of TDSs with CPE and arbitrarily high entropy rank have been constructed in \cite{barbieri2020,salo2019entropy,darjiGR}.
\bibliographystyle{plain}
\bibliography{references}

\end{document}